\documentclass[final,1p,times]{elsarticle}
\usepackage{amsmath}
\allowdisplaybreaks[4]
\usepackage{amsthm}
\usepackage{amscd}
\RequirePackage{CJKnumb}
\usepackage{amsfonts}
\usepackage{amssymb}
\usepackage{graphicx}
\usepackage{caption} 
\usepackage{subfigure}
\usepackage{subcaption}
\usepackage{booktabs}
\biboptions{sort&compress}
\usepackage{xcolor}
\usepackage[pagebackref, colorlinks, citecolor=green, linkcolor=red, urlcolor=blue]{hyperref}
\numberwithin{equation}{section}
\newtheorem{theorem}{Theorem}[section]
\newtheorem{lemma}[theorem]{Lemma}
\usepackage{mathrsfs}
\usepackage{titletoc}
\usepackage{float}
\usepackage{longtable}
\usepackage{adjustbox}
\journal{***}
\begin{document}
	\begin{frontmatter}
		\title{Existence and asymptotic analysis of topological solutions for generalized Chern--Simons equations on discrete lattice graphs}
		\author{Songbo Hou \corref{cor1}}
		\ead{housb@cau.edu.cn}
		\address{Department of Applied Mathematics, College of Science, China Agricultural University,  Beijing, 100083, P.R. China}
		
		\begin{abstract}
		We study a class of generalized Chern-Simons equations on discrete lattice graphs. By an iterative scheme combined with an exhaustion argument,  we  establish the existence of topological solutions, which is also the maximal topological solution. We further examine the behavior of the maximal topological solution as the parameter tends to either infinity or zero. The present work extends the results of Hua et al., arXiv:2310.13905
		(2023) and Hou and Kong, Calc. Var. Partial Differ. Equ. 64(3), 77 (2025).

		\end{abstract}	
	\end{frontmatter}
	
	\section{Introduction}
	
The Chern-Simons equation originates from three-dimensional topological field theory and serves as a fundamental model in theoretical physics for exploring the interaction between gauge fields and topological structures. Initially introduced by S.-S. Chern and J. Simons, the equation provides a key Lagrangian framework for constructing three-dimensional gauge theories, and has found broad applications in mathematical physics, gauge field theory, and condensed matter physics. Its core idea lies in using the Chern-Simons action to describe the geometric and topological properties of gauge fields, thereby revealing the intrinsic symmetries of the system in a manner independent of the underlying metric.

In mathematical studies, the Chern-Simons equation is closely related to gauge theory \cite{MR2177747}, differential geometry \cite{MR0353327}, nonlinear partial differential equations \cite{MR1324400}, and harmonic map theory \cite{MR1204317}. It exhibits rich analytical features, particularly in the investigation of vortex-type solutions, multi-solution structures, stability analysis, and energy estimates.

Recently, the study of Chern--Simons equations has been extended from continuous domains to discrete graphs. Huang et al.~\cite{huang2020existence} investigated the following equation
\begin{equation}\label{1.1}
	\Delta f = \lambda e^f(e^f - 1) + 4\pi \sum_{j=1}^M \delta_{p_j}
\end{equation}
on finite graphs, where $\lambda>0$ denotes a parameter, $M$ is a positive integer, $p_j$ represents a vertex of the graph, and $\delta_{p_j}$ stands for the Dirac delta mass centered at $p_j$. It was shown that a critical parameter $\lambda_c$ exists such that Eq.~\eqref{1.1} has a solution for $\lambda>\lambda_c$, whereas no solution occurs for $\lambda<\lambda_c$. In~\cite{hou2022existence}, Hou and Sun studied a class of generalized Chern--Simons equations on finite graphs and applied their framework to Eq.~\eqref{1.1}, proving that a solution also exists when $\lambda = \lambda_c$. 

Using topological degree theory, Li et al.~\cite{li2023topological} reestablished that Eq.~\eqref{1.1} on finite graphs admits multiple solutions.
 On infinite lattice graphs, Hua et al.~\cite{hua2023existence} studied the existence and decay behavior of topological solutions to Eq.~\eqref{1.1}.

Chen and Han~\cite{MR3107578} considered \eqref{1.2} on a doubly periodic domain in $\mathbb{R}^2$ and showed that the solvability depends on the parameter $\lambda$.
Subsequently, as a natural extension of \eqref{1.1}, Gao and Hou~\cite{gao2022existence} studied
\begin{equation}\label{1.2}
	\Delta f = \lambda e^f\bigl(e^{a f} - 1\bigr) + 4\pi \sum_{j=1}^M n_j \delta_{p_j},
\end{equation}
where $a>0$, each $n_j$ is a positive integer, and $\delta_{p_j}$ is defined as in \eqref{1.1}.
They established existence results and further proved the multiplicity of solutions.

The study of Chern--Simons models on discrete graphs also encompasses Chern--Simons systems. Huang et al.~\cite{huang2021mean} investigated a class of Chern--Simons systems, establishing the existence of a maximal solution and further proving the multiplicity of solutions, including one that is a local minimizer of the associated energy functional and another of mountain-pass type. Chao et al.~\cite{chao2022existence} extended the work of~\cite{huang2021mean} by employing the method of upper and lower solutions combined with a priori estimates to analyze a class of generalized Chern--Simons systems.

Further studies on Chern--Simons models on discrete graphs include~\cite{CHAO2023126787, MR4781956, hou2024topological, lu2021existence, MR4725973, MR4669943, MR4873677}, which explore various existence, uniqueness, and asymptotic properties of solutions.

	This work focuses on the discrete analogue of Eq.~\eqref{1.2} posed on infinite graphs.  Before that, we first introduce some basic concepts of graphs.  Let \( G = (V, E) \) denote a graph with vertex set \( V \) and edge set \( E \).  Each edge $xy\in E $ is assigned a weight $\omega_{xy}$, which is positive and symmetric. In the following, we consider lattice graphs, which are a special class of graphs. Their vertex set $V$ consists of all vectors 
$x= (x_1, \ldots, x_n)  $ whose components are integers.  Their edge set $E$ consists of all pairs $x y$ such that $d(x, y)=1$, where

$$
d(x, y)=\sum_{i=1}^n\left|x_i-y_i\right|
$$
denotes the lattice (Manhattan) distance between $x$ and $y$. For $n \geq 2$,  we denote the lattice graph by $\mathbb{Z}^n=(V, E)$, where the weight on each edge $x y$ is $\omega_{x y}=1$ if $x y \in E$. We also write $y \sim x$ to indicate that $y$ is adjacent to $x$. Let \( \Omega \subset \mathbb{Z}^n \) be a finite subset.  
	The boundary of \( \Omega \), denoted by \( \partial \Omega \), is defined as the set of all points not in \( \Omega \) but at distance one from some point of \( \Omega \).
	 We also denote the closure of $\Omega$ by $\bar{\Omega}=\Omega \cup \partial \Omega$.

Next, we define the operators and function spaces on \( \mathbb{Z}^n \) that will be used throughout the paper.
 For a subset $\Omega \subset V$, let $C(\Omega)$ denote the set of all real-valued functions defined on $\Omega$. The measure on $V$ is taken to be uniform, with $\mu(x)=1$ for all $x \in V$. For any $f \in C(\Omega)$, its integral over \( \Omega \) is defined as
$$
\int_{\Omega} f d \mu=\sum_{x \in \Omega} f(x).
$$

Similar to the Euclidean setting, for $1 \leq p<\infty$, we can define the space $L^p(V)$, consisting of realvalued functions on $V$ with the norm

$$
\|f\|_p=\left(\int_{V} |f(x)|^p d \mu\right)^{\frac{1}{p}}.
$$
Similarly, we define the space $L^{\infty}(V)$, with the norm given by
$$
\|f\|_{\infty}=\sup _{x \in V}|f(x)| .
$$

For any $f \in C(V)$, we define the Laplacian operator as

$$
\Delta f(x)=\sum_{y \sim x}(f(y)-f(x)).
$$
We define the discrete gradient along an edge $xy$ as $\nabla f(x, y) = f(y) - f(x)$. Then the pointwise inner product of gradients is given by
\[
\langle \nabla f, \nabla g \rangle(x) = \frac{1}{2} \sum_{y \sim x} \nabla f(x, y) \nabla g(x, y).
\]
When $f=g$, we denote

$$
|\nabla f|^2(x)=\frac{1}{2} \sum_{y \sim x}(f(y)-f(x))^2,
$$
which represents the squared norm of the discrete gradient of $f$ at the vertex $x$.

Let $d(x)=d(x, 0)$ denote the distance from the origin. In this paper, we study Eq.~\eqref{1.2} on the lattice $\mathbb{Z}^n$ with 
$$
\delta_{p_j}(x)= \begin{cases}1, & \text { if } x=p_j, \\ 0, & \text { if } x \neq p_j.\end{cases}
$$
We are interested in topological solutions, namely, solutions satisfying $f(x) \rightarrow 0$ as $d(x) \rightarrow \infty$. The first result of this paper is presented below.
\begin{theorem}\label{th1}
There exists a topological solution \( f_{\lambda} \in L^{p}(V) \) for all \( 1 \leq p \leq \infty \) to Eq.~\eqref{1.2}, which is also the maximal solution. Moreover, for any $0<\varepsilon<1$, the following decay estimate holds:
\[
f_{\lambda}(x) = O\left(e^{-\alpha(1 - \varepsilon) d(x)}\right),
\]
where \( \alpha = \ln \left(1 + \frac{\lambda a}{2n} \right) \).

\end{theorem}

As the parameters vary, the asymptotic behaviors of solutions to Chern–Simons systems and equations on discrete graphs have recently been studied. On finite graphs, Hou and Kong~\cite{houkong2022existence} established the existence of solutions and derived their asymptotic behavior as the  parameter $\lambda\to\infty$. Subsequently, Liu~\cite{liu2025existence} investigated a skew-symmetric Chern--Simons system on lattice graphs and derived precise asymptotic behaviors of topological solutions as $\lambda$ approaches either $+\infty$ or $0^{+}$.

For the infinite lattice $\mathbb{Z}^n$, let $G_n$ denote the Green's function of the discrete Laplacian. It is uniquely determined by
\[
\Delta G_n=\delta_0 \quad\text{in }\mathbb{Z}^n,\qquad 
\lim_{d(x)\to\infty} G_n(x)=0.
\]
A well-known Fourier representation of $G_n$ is
\[
G_n(x)= -\frac{1}{(2\pi)^n}\int_{[-\pi,\pi]^n}
\frac{e^{i z\cdot x}}{\,2n - 2\sum_{j=1}^{n}\cos z_j\,}\,dz,
\]
where $x\in\mathbb{Z}^n$ and $z=(z_1,\ldots,z_n)\in\mathbb{R}^n$ \cite{MR4448687}. 

Motivated by the approach developed in~\cite{liu2025existence}, we establish the following result.

\begin{theorem}\label{th2}
	Let $f_{\lambda}$ be the maximal topological solution to Eq.~\eqref{1.2}. Then the following asymptotic behaviors hold:
	\begin{enumerate}
		\item[(i)] $f_{\lambda} \to 0$ \quad as \quad $\lambda \to \infty$;
		\item[(ii)] if $n=2$, then $f_{\lambda} \to -\infty$ \quad as \quad $\lambda \to 0^+$;
		\item[(iii)] if $n=3$, then 
$f_{\lambda} \to 4 \pi \sum\limits_{j=1}^{M} n_j G_n(x-p_j) \quad \text{as} \quad \lambda \to 0^+.
		$
	\end{enumerate}
\end{theorem}

In Section 2, we prove Theorem \ref{th1}. We begin by proving that the equation admits a solution in a bounded domain \( \Omega \). To achieve this, we employ an iterative method beginning with the initial function $f_0=0$, which generates a monotone decreasing sequence $\left\{f_k\right\}$. If this sequence converges, its limit yields a solution to the equation on $\Omega$.

To investigate convergence, we introduce an associated energy functional $I_{\bar{\Omega}}(f_k)$, and show that it is monotone decreasing and bounded from above. Based on this functional, we derive estimates that imply the uniform boundedness of the sequence $\left\|u_k\right\|_{L^{2}(\Omega)}$. Since $L^{2}(\Omega)$ is finite- dimensional, this ensures the existence of a solution on $\Omega$.

We then consider an increasing sequence of domains $\Omega_0 \subset \Omega_1 \subset \cdots \subset \Omega_i \subset \cdots$, and examine the corresponding sequence of solutions. We prove that this sequence converges, and its limit defines a global solution on $V$, which is also a topological solution. 

In section 3, we prove Theorem \ref{th2}.  We first prove that for sufficiently large $\lambda$, $f_{\lambda}$ is bounded below. Furthermore, this lower bound tends to zero as $\lambda \to \infty$. Combining this with the fact that $f_{\lambda} \le 0$, we conclude that $f_{\lambda} \to 0$ as $\lambda \to \infty$.

In the case $n=2$, we  first assume that $f_{\lambda}$ converges to a limit function $f$ as $\lambda \to 0^+$. A contradiction is then derived by analyzing the solution on finite cubes, which implies that the  assumption is false. Consequently, there must exist a point $x_0$ such that $f_{\lambda}(x_0) \to -\infty$ as $\lambda \to 0^+$. Finally, utilizing the connectivity of the graph $\mathbb{Z}^2$ and local discrete estimates, we show that this blow-up holds for  all points.

In the case $n=3$, we define the function
$
\phi_n(x)=4 \pi \sum\limits_{j=1}^{M} n_j G_n(x-p_j).
$
We first establish the uniform bounds $\phi_n(x) \le f_\lambda(x) \le 0$. Assuming that $f(x) = \lim_{\lambda\rightarrow 0^+}f_\lambda(x)$,   we  get 
$
f(x) = \phi_n(x).
$

Our method is inspired by the approaches developed in \cite{MR1320569,hua2023existence,hou2024topological,liu2025existence}. 

\section{Proof of Theorem \ref{th1}}

	In this section, we employ variational methods to prove Theorem \ref{th1}. We begin by considering the Cauchy problem for Eq.~\eqref{1.2} on a bounded domain and establish the existence of a solution. We construct a sequence of functions via an iterative scheme.
	
	We first choose a bounded domain $\Omega_{0}$ that contains all the singularities $\{p_j\}$. Then, we select a larger bounded and connected domain $\Omega$ such that $\Omega_{0} \subset \Omega$.
	
	Let
	$$
	g=4 \pi \sum_{j=1}^M n_j \delta_{p_j}, \quad \text { and } \quad N=4 \pi \sum_{j=1}^M n_j.
	$$

	Fix a constant $K>a \lambda$. Starting from the initial value $f_0=0$, we define the following iteration scheme:
	
\begin{equation}\label{2.1}
	\begin{cases}(\Delta-K) f_k=\lambda e^{f_{k-1}}(e^{af_{k-1}}-1)+g-K f_{k-1}, & \text { in } \Omega, \\ f_k=0, & \text { on } \partial \Omega.\end{cases}
\end{equation}

We next present the following lemma:
	
	\begin{lemma}\label{lemma 2.1}
Let $\left\{f_k\right\}$ denote the sequence determined by the iteration scheme (\ref{2.1}). Then the sequence is monotone decreasing and satisfies
$$
0=f_0 \geq f_1 \geq f_2 \geq \cdots.
$$
\end{lemma}	
\begin{proof}
We denote by \( C_0(\bar{\Omega}) \) the space of functions defined on \( \bar{\Omega} \) that vanish on the boundary \( \partial \Omega \).
    We begin by showing that, for each \( v \in C_0(\bar{\Omega}) \), the boundary value problem
    \begin{equation}\label{2.2}
    	\begin{cases}
    		(\Delta - K)u = v, & \text{in } \Omega, \\
    		u = 0, & \text{on } \partial \Omega,
    	\end{cases}
    \end{equation}
    has a solution \( u \in C_0(\bar{\Omega}) \). The argument follows the same lines as the proof of Theorem~2.1 in~\cite{ge2020p} and Lemma~2.2 in~\cite{houkong2022existence}.
     Let us introduce the associated functional

$$
F(u)=\frac{1}{2} \int_{\bar{\Omega}}|\nabla u|^2 d \mu+\frac{1}{2} \int_{\Omega} K u^2 d \mu+\int_{\Omega} v u d \mu, \quad u \in C_0(\bar{\Omega}),
$$
where $$\int_{\bar{\Omega}}|\nabla u|^2 d \mu=\frac{1}{2}\sum\limits _{\substack{ x,y\in  \bar{\Omega}\\x\sim y}}(u(y)-u(x))^2.$$
Any critical point of the functional $F(u)$ yields a solution to Eq.~\eqref{2.2}. For any test function $\phi \in C_0(\bar{\Omega})$, a straightforward calculation yields
$$
\left.\frac{d}{d t}\right|_{t=0} F(u+t \phi)=-\int_{\Omega}(\Delta u-K u-v) \phi d \mu.
$$
Therefore, the condition
\[
\left.\frac{d}{dt}\right|_{t=0} F(u + t\phi) = 0
\]
is satisfied for every \( \phi \in C_0(\bar{\Omega}) \) precisely when \( u \) solves
\[
\Delta u - K u = v.
\]
Thus, solving Eq.~\eqref{2.2} reduces to minimizing the functional \( F(u) \).
 Note that, by the Cauchy-Schwarz inequality, it follows that
$$
\left|\int_{\Omega} v u d \mu\right| \leq\|v\|_{L^2(\Omega)}\|u\|_{L^2(\Omega)}.
$$ 
We can obtain the following inequality:
$$
F(u) \geq \frac{1}{2} \int_{\bar{\Omega}}|\nabla u|^2 d \mu+\frac{1}{2} \int_{\Omega} K u^2 d \mu-\|v\|_{L^2(\Omega)}\|u\|_{L^2(\Omega)}
$$
which leads to
$$
F(u) \geq \frac{K}{2}\|u\|^2_{L^2(\Omega)}-\|v\|_{L^2(\Omega)}\|u\|_{L^2(\Omega)}
$$
As a result, we have
$$
E(u) \rightarrow+\infty \quad \text { as }  \quad\|u\|_{L^2(\Omega)}^2 \rightarrow+\infty.
$$
Note that $\|u\|_{L^2(\Omega)}^2 \rightarrow+\infty$ is equivalent to $\sup _{x \in \bar{\Omega}}|u(x)| \rightarrow+\infty$, and that $C_0(\bar{\Omega})$ is finite-dimensional. It follows that the functional $F(u)$ achieves its minimum at a function $u \in C_0(\Omega)$, and this function is a solution to Eq.~\eqref{2.2}.

To proceed, we begin by analyzing the iteration process introduced earlier. It can be seen that $f_1$ satisfies the boundary value problem:
\begin{equation}\label{2.3}
\begin{cases}(\Delta-K) f_1=g, & \text { in } \Omega, \\ f_1=0, & \text { on } \partial \Omega .\end{cases}
\end{equation} Noting that $g \in C_0(\bar{\Omega})$, the existence of $f_1$ follows directly from the previous claim. Furthermore, the maximum principle (Lemma 2.2 in \cite{hua2023existence}) guarantees $f_1\leq 0$. 
By induction, assume that
$$
0=f_0 \geq f_1 \geq f_2 \geq \cdots \geq f_{k-1}.
$$
Observe that the expression on the right of Eq.~(\ref{2.1}),
\[
\lambda e^{f_{k-1}}(e^{a f_{k-1}} - 1) + g - K f_{k-1},
\]
lies in \( C_0(\bar{\Omega}) \), the existence of $f_k$ then follows from  the previous claim.

Using the mean value theorem, we obtain
$$\begin{aligned}(\Delta-K)\left(f_{k}-f_{k-1}\right) & =\lambda e^{f_{k-1}}(e^{af_{k-1}}-1)-\lambda e^{f_{k-2}}(e^{af_{k-2}}-1)-K\left(f_{k-1}-f_{k-2}\right) \\ & \geqslant \lambda e^{\xi}[(a+1)e^{a\xi}-1](f_{k-1}-f_{k-2})-K\left(f_{k-1}-f_{k-2}\right) \\ & \geqslant (\lambda a-K)(f_{k-1}-f_{k-2}) .\\ & \geqslant 0,\end{aligned}$$ 
where $\xi$ lies between $f_{k-1}$ and $f_{k-2}$, i.e., $f_{k-1} \leq \xi \leq f_{k-2}$.

Applying the maximum principle again, we conclude that $f_k \leq f_{k-1}$, thereby completing the proof.
\end{proof}	
	
We now construct the functional  corresponding to Eq.~\eqref{1.2}:
$$
	I_{\bar{\Omega}}(f)=\frac{1}{2} \int_{\bar{\Omega}}|\nabla f|^2 d \mu+\frac{\lambda}{a+1} \int_{\Omega}(e^{(a+1)f}-1 )d \mu+\lambda \int_{\Omega} (1-e^{f}) d \mu+\int_{\Omega}gfd\mu.
$$  The critical points of this functional in the space $C_0(\bar{\Omega})$ will yield solutions to Eq.~\eqref{1.2} on the domain $\Omega$.	

	The following theorem shows that applying the functional to the sequence $\left\{f_k\right\}$ obtained in Lemma \ref{lemma 2.1} yields a monotone decreasing sequence, and it also provides an upper bound estimate.
	
\begin{lemma}\label{Lemma 2.2}
Let \( \{f_k\} \) denote the sequence obtained from the iteration~\eqref{2.1}.
 Then we obtain the following inequality:

$$
0 \geqslant I_{\bar{\Omega}}(f_1) \geqslant I_{\bar{\Omega}}(f_2) \cdots \geqslant I_{\bar{\Omega}}(f_k) \geqslant \cdots.
$$

\end{lemma}
\begin{proof}
We first show that the sequence \( \{ I_{\bar{\Omega}}(f_k) \} \) is monotone.
 To begin with, applying Lemma 2.2 from \cite{MR3833747}, we obtain
\begin{equation}\label{2.4}
\int_{\bar{\Omega}}\langle \nabla f_k, \nabla (f_k-f_{k-1})\rangle d\mu=- \int_{\Omega}\Delta f_k(f_k-f_{k-1})d\mu.
\end{equation}
From Eq.~\eqref{2.1}, it follows that
\begin{equation}\label{2.5}
\int_{\Omega}(\Delta-K)f_k(f_k-f_{k-1})d\mu=\int_{\Omega}\left[\lambda e^{f_{k-1}}(e^{af_{k-1}}-1)+g-Kf_{k-1}\right](f_k-f_{k-1})d\mu.
\end{equation}
 Combining equations (\ref{2.4}) and (\ref{2.5}), we arrive at the following result: 
\begin{equation}\label{2.6}
	\begin{aligned}
		&\int_{\bar{\Omega}} |\nabla f_k|^2  d\mu - \int_{\bar{\Omega}} \langle \nabla f_k, \nabla f_{k-1} \rangle  d\mu 
		+ K \int_{\Omega} (f_k - f_{k-1})^2  d\mu \\
		&= - \int_{\Omega} \left[ \lambda e^{f_{k-1}}(e^{a f_{k-1}} - 1) + g \right] (f_k - f_{k-1}) d\mu.
	\end{aligned}
\end{equation}
Note that 
$$\left|\int_{\bar{\Omega}} \langle \nabla f_k, \nabla f_{k-1} \rangle  d\mu \right|\leq \frac{1}{2}\int_{\bar{\Omega}} |\nabla f_k|^2d\mu +\frac{1}{2}\int_{\bar{\Omega}} |\nabla f_{k-1}|^2d\mu.$$
Subsequently, Eq.~\eqref{2.6} implies that
\begin{equation}\label{2.7}
	\begin{aligned}
		\frac{1}{2} \int_{\bar{\Omega}} |\nabla f_k|^2 d\mu 
		&\leq \frac{1}{2} \int_{\bar{\Omega}} |\nabla f_{k-1}|^2 \, d\mu 
		- K \int_{\Omega} (f_k - f_{k-1})^2 d\mu \\
		&\quad - \int_{\Omega} \left[ \lambda e^{f_{k-1}}(e^{a f_{k-1}} - 1) + g \right] (f_k - f_{k-1})  d\mu.
	\end{aligned}
\end{equation}

  Here we construct an auxiliary function:
$$h(x)=\frac{\lambda}{a+1}e^{(a+1)x}-\lambda e^x-\frac{K}{2}x^2.$$
Note that $K>a \lambda$. It is straightforward to verify that $h(x)$ is concave for $x \leq 0$. Hence, we obtain
$$
h\left(f_{k-1}\right)-h\left(f_k\right) \geqslant h^{\prime}\left(f_{k-1}\right)\left(f_{k-1}-f_k\right)=\left[\lambda e^{f_{k-1}}(e^{a f_{k-1}} - 1) -Kf_{k-1}\right]\left(f_{k-1}-f_k\right) .
$$
It follows that
\begin{equation}\label{2.8}
	\begin{aligned}
		\frac{\lambda}{a+1}e^{(a+1)f_k}-\lambda e^{f_k}\leqslant & \frac{\lambda}{a+1}e^{(a+1)f_{k-1}}-\lambda e^{f_{k-1}}+\frac{K}{2}\left(f_k-f_{k-1}\right)^2 \\
		& +\lambda e^{f_{k-1}}(e^{a f_{k-1}} - 1)\left(f_k-f_{k-1}\right) .
	\end{aligned}
\end{equation}
From equations (\ref{2.7}) and (\ref{2.8}), we obtain
$$
I_{\bar{\Omega}}\left(f_k\right) \leqslant I_{\bar{\Omega}}\left(f_k\right)+\frac{K}{2}\left\|f_{k-1}-f_k\right\|_{L^2(\Omega)}^2 \leqslant I_{\bar{\Omega}}\left(f_{k-1}\right) .
$$	
Noting that  $I_{\bar{\Omega}}\left(f_{0}\right)=0$, we finish the proof.
\end{proof}

We next show that \( \{u_k\} \) remains bounded in \( L^2(\Omega) \), where the bound does not depend on the domain \( \Omega \).
We first present a lemma for later use.
\begin{lemma}\label{le2.3}
For \( x \leq 0 \) and \( a > 0 \), there exists a positive constant \( c \) such that
\[
\frac{(e^{(a+1)x} - 1) + (a+1)(1 - e^x)}{a+1} \geq c \left( \frac{|x|}{1 + |x|} \right)^2.
\]
\end{lemma}

\begin{proof}

Set \(k=a+1\), \(t=-x\) and \(s=e^{-t}\in(0,1]\).  Since \(e^{-kt}=s^{k}\) and
\(1-e^{-t}=1-s\), we obtain
\[
g(t):=\frac{e^{-kt}-1}{k}+1-e^{-t}
=\frac{s^{k}-1}{k}+1-s
=\frac{s^{k}-ks+(k-1)}{k}.
\]
Define
$
f(s):=s^{k}-ks+k-1,
$ so that 
$g(t)=\frac{f(s)}{k}.
$

Let 
\[
f(s) = (1 - s)^2 h(s),
\]
where \( h(s) = \dfrac{f(s)}{(1 - s)^2} \). For all \( s \in (0,1) \), we have
\[
f'(s) = k(s^{k-1} - 1) < 0,
\]
so \( f \) is strictly decreasing on \( (0,1) \). Since \( f(1) = 0 \), it follows that \( f(s) > 0 \) for all \( s \in (0,1) \). Consequently,
\[
h(s) = \dfrac{f(s)}{(1 - s)^2} > 0 \quad \text{for all } s \in (0,1).
\]

Near \(s=1\), a second--order Taylor expansion yields
	\[
	s^{k}\approx 1-k(1-s)+\frac{k(k-1)}{2}(1-s)^{2}
	\]
	which implies that \[f(s)\approx\frac{k(k-1)}{2}(1-s)^{2}>0.\]

Note that 
\[
h(0)=\frac{f(0)}{1}=k-1>0,\qquad
\lim_{s\to1^{-}}h(s)=\frac{k(k-1)}{2}>0.
\]
Define \( h(1) := \dfrac{k(k-1)}{2} \) so that \( h \) is continuous and strictly positive on \([0,1]\).

It follows from the extreme–value theorem that there exists \( m > 0 \) satisfying
\[
h(s)\ge m \quad\text{for all }s\in[0,1].
\]

Since \(1-s=1-e^{-t}\),  we get
\[
g(t)=\frac{(1-e^{-t})^{2}\,h(s)}{k}.
\]

For \(t>0\), the elementary inequality \(e^{t}>1+t\) implies
\[
e^{-t}<\frac{1}{1+t},
\] which yields \[1-e^{-t}>\frac{t}{1+t}.\]
Consequently,
\[
(1-e^{-t})^{2}>\left(\frac{t}{1+t}\right)^{2}.
\]

Note that 
\[
g(t)=\frac{(1-e^{-t})^{2}\,h(s)}{k}
>\frac{m}{k}\left(\frac{t}{1+t}\right)^{2}
=:c\left(\frac{t}{1+t}\right)^{2},
\qquad c:=\frac{m}{k}>0.
\]
Therefore
\[
\frac{(e^{(a+1)x} - 1) + (a+1)(1 - e^x)}{a+1} \geq c \left( \frac{|x|}{1 + |x|} \right)^2,
\]
establishing the desired lower bound controlled by
\(\left(\tfrac{|x|}{1+|x|}\right)^{2}\).
\end{proof}

Next, we estimate the bound of $\left\|f_k\right\|_{L^2(\Omega)}$ by analyzing the terms appearing in the expression of $I_{\bar{\Omega}}\left(f_k\right)$.
\begin{lemma}\label{le2.4}
Let $\{f_k \}$ be the sequence determined by Lemma {\ref{lemma 2.1}}. As a result, one obtains the estimate
\begin{equation}\label{2.11}
	\|f_k\|_{L^{2}(\Omega)} \leq C^{\prime}\!\left(I_{\bar{\Omega}}(f_k) + C^{\prime\prime}\right) \leq C,
\end{equation}
with constants \( C^{\prime} \), \( C^{\prime\prime} \), and \( C \) depending only on \( \lambda \), \( a \), \( n \), and \( N \).
\end{lemma}
\begin{proof}

We begin by introducing an inequality.  For $u\in L^{p}(V)\, (p>1)$, define 
$$
|u|_{1, p}:=\left(\sum_{x \in V} \sum_{y \sim x}|u(y)-u(x)|^p\right)^{\frac{1}{p}}.
$$In the proof of Theorem 4.1 in \cite{MR4095474}, Porretta established the following estimate:

\begin{equation}\label{2.12}
\|u\|_{L^{\frac{\gamma n}{n-1}}(V)}^{\gamma} \leq C(p,n,\gamma)\|u\|_{L^{(\gamma-1)p^{\prime}}(V)}|u|_{1, p},
\end{equation}
where \( p>1 \), \( \gamma \geqslant p \), \( p^{\prime} = \frac{p}{p-1} \), and \( c(p,n,\gamma) > 0 \) is a constant that depends solely on \( p \), \( \gamma \), and \( n \).
By setting \( p = \gamma = 2 \) in~\eqref{2.12}, we derive
\begin{equation}\label{2.13}	
	\|u\|_{L^4(V)}^2 \leq \|u\|_{L^{\frac{2n}{n-1}}(V)}^{2} \leq C(n)\|u\|_{L^2(V)}|u|_{1,2},
\end{equation}
where \( C(n) \) denotes a constant depending only on \( n \).
Here the inequality \( \|u\|_{L^4(V)} \leq \|u\|_{L^{\frac{2n}{n-1}}(V)} \) has been applied (see Lemma~2.1 in~\cite{MR3277179}).

Now we extend $f_k(x)$ to the entire space $V$:
\begin{equation*}
	\tilde{f}_k(x) = \left\{
	\begin{array}{l@{\quad\text{on}\ }l}
		f_k(x) & \Omega, \\
		0 & V \backslash \Omega.
	\end{array}
	\right.
\end{equation*}
Note that $f_k=0$ on $\partial \Omega$, we get 
\begin{equation}\label{2.14}
|\tilde{f}_k|_{1,2} \leqslant \left(2 \int_{V}|\nabla \tilde{f_k}|^2 d\mu\right)^{\frac{1}{2}} =\left(2 \int_{\bar{\Omega}}|\nabla f_k|^2 d\mu\right)^{\frac{1}{2}}.
\end{equation}
It follows from (\ref{2.13}) and (\ref{2.14}) that 

\[\|\tilde{f}_k\|_{L^4(\Omega)}^4\leq C_1\|\tilde{f}_k\|_{L^2(\Omega)}^2 \|\nabla f_k\|^2_{L^2(\bar{\Omega})},\]
where $C_1=2C^2(n)$. Hence \begin{equation}\label{2.15}
\|f_k\|_{L^4(\Omega)}^4\leq C_1\|f_k\|_{L^2(\Omega)}^2 \|\nabla f_k\|^2_{L^2(\bar{\Omega})}.\end{equation}

Using Lemma {\ref{le2.3}}  and (\ref{2.15}), we conclude that 
\begin{equation}\label{2.16}
	\begin{aligned}
I_{\bar{\Omega}}(f_k)&\geqslant\frac{1}{2} \|\nabla f_k\|^2_{L^2(\bar{\Omega})}+\lambda c\int_{\Omega}\left(\frac{|f_k|}{1+|f_k|}\right)^2d\mu+\int_{\Omega}gf_kd\mu\\
	& \geqslant\frac{1}{2} \|\nabla f_k\|^2_{L^2(\bar{\Omega})}+\lambda c\int_{\Omega}\left(\frac{|f_k|}{1+|f_k|}\right)^2d\mu-\|g\|_{L^{\frac{4}{3}}(\Omega)} \|f_k\|_{L^4(\Omega)}\\
	& \geqslant\frac{1}{2} \|\nabla f_k\|^2_{L^2(\bar{\Omega})}+\lambda c\int_{\Omega}\left(\frac{|f_k|}{1+|f_k|}\right)^2d\mu -C_2\|f_k\|_{L^2(\Omega)}^{\frac{1}{2}} \|\nabla f_k\|^{\frac{1}{2}} _{L^2(\bar{\Omega})}\\
	&\geqslant\frac{1}{2} \|\nabla f_k\|^2_{L^2(\bar{\Omega})}+\lambda c\int_{\Omega}\left(\frac{|f_k|}{1+|f_k|}\right)^2d\mu -\epsilon \|f_k\|_{L^2(\Omega)}-\frac{C_2^2}{4\epsilon }\|\nabla f_k\|_{L^2(\bar{\Omega})}\\
	&\geqslant\frac{1}{2} \|\nabla f_k\|^2_{L^2(\bar{\Omega})}+\lambda c\int_{\Omega}\left(\frac{|f_k|}{1+|f_k|}\right)^2d\mu -\epsilon \|f_k\|_{L^2(\Omega)}-\frac{1}{4}\|\nabla f_k\|^2_{L^2(\bar{\Omega})}-C_3\\
	&=\frac{1}{4} \|\nabla f_k\|^2_{L^2(\bar{\Omega})}+\lambda c\int_{\Omega}\left(\frac{|f_k|}{1+|f_k|}\right)^2d\mu -\epsilon \|f_k\|_{L^2(\Omega)}-C_3,
\end{aligned}
\end{equation}
where $C_2$ only depending on $n$ and $N$, $C_3=\frac{C_2^4}{16\epsilon^2}$.

Applying (\ref{2.15}), we arrive at the following estimate:
\begin{align*}
	&\left(\int_{\Omega}f_k^2d\mu\right)^2=\left[\int_{\Omega}\frac{|f_k|}{1+|f_k|}(1+|f_k|)|f_k|d\mu\right]^2\\
	\leqslant & \int_{\Omega}\left(\frac{|f_k|}{1+|f_k|}\right)^2d\mu\int_{\Omega}
	(1+|f_k|)^2|f_k|^2d\mu\\
	\leqslant & 2\int_{\Omega}\left(\frac{|f_k|}{1+|f_k|}\right)^2d\mu\int_{\Omega}
	(|f_k|^2+|f_k|^4)d\mu\\
	\leqslant &2\|f_k\|_{L^2(\Omega)}^2\int_{\Omega}\left(\frac{|f_k|}{1+|f_k|}\right)^2d\mu+2C_1\|f_k\|_{L^2(\Omega)}^2 \|\nabla f_k\|^2_{L^2(\bar{\Omega})}\int_{\Omega}\left(\frac{|f_k|}{1+|f_k|}\right)^2d\mu \\
	\leqslant &\frac{1}{4}\|f_k\|_{L^2(\Omega)}^4+4\left[\int_{\Omega}\left(\frac{|f_k|}{1+|f_k|}\right)^2d\mu \right]^2+\frac{1}{4}\|f_k\|_{L^2(\Omega)}^4+4C_1^2\|\nabla f_k\|^4_{L^2(\bar{\Omega})}\left[\int_{\Omega}\left(\frac{|f_k|}{1+|f_k|}\right)^2d\mu \right]^2\\
	=&\frac{1}{2}\|f_k\|_{L^2(\Omega)}^4+4\left[\int_{\Omega}\left(\frac{|f_k|}{1+|f_k|}\right)^2d\mu \right]^2+4C_1^2\|\nabla f_k\|^4_{L^2(\bar{\Omega})}\left[\int_{\Omega}\left(\frac{|f_k|}{1+|f_k|}\right)^2d\mu \right]^2\\
	\leqslant &\frac{1}{2}\|f_k\|_{L^2(\Omega)}^4+C_4 \left\{ \left[\int_{\Omega}\left(\frac{|f_k|}{1+|f_k|}\right)^2d\mu \right]^2+\|\nabla f_k\|^4_{L^2(\bar{\Omega})}\left[\int_{\Omega}\left(\frac{|f_k|}{1+|f_k|}\right)^2d\mu \right]^2\right\}\\
	\leqslant & \frac{1}{2}\|f_k\|_{L^2(\Omega)}^4+C_4 \left\{1+ \left[\int_{\Omega}\left(\frac{|f_k|}{1+|f_k|}\right)^2d\mu \right]^4+\|\nabla f_k\|^8_{L^2(\bar{\Omega})}\right\}\\
	\leqslant & \frac{1}{2}\|f_k\|_{L^2(\Omega)}^4+C_4 \left\{1+ \left[\int_{\Omega}\left(\frac{|f_k|}{1+|f_k|}\right)^2d\mu \right]+\|\nabla f_k\|_{L^2(\bar{\Omega})}^2\right\}^4,
\end{align*}
where $C_4=4(1+C_1^2)$.  Hence, we have 
\begin{equation}\label{2.17}
\|f_k\|_{L^2(\Omega)}\leq C_5\left[ 1+\int_{\Omega}\left(\frac{|f_k|}{1+|f_k|}\right)^2d\mu+\|\nabla f_k\|^2_{L^2(\bar{\Omega})}\right],\end{equation}
where $C_5$ is only depending on $n$. 

By choosing
$$
\epsilon=\frac{\min \left\{\frac{1}{8}, \frac{\lambda c}{2}\right\}}{C_5},
$$
and applying estimates ( \ref{2.16}) and (\ref{2.17} ), we obtain
$$
\left\|f_k\right\|_{L^2(\Omega)} \leq C^{\prime}\left(I_{\bar{\Omega}}\left(f_k\right)+C^{\prime\prime}\right)
$$
In view of Lemma \ref{Lemma 2.2}, it follows that
$$
\left\|f_k\right\|_{L^2(\Omega)} \leq C.
$$
\end{proof}

Applying Lemma~\ref{le2.4}, we establish the solvability of the boundary value problem associated with Eq.~\eqref{1.2} on a bounded domain.

\begin{lemma}\label{le2.5}
Consider the following boundary value problem:
\begin{equation}\label{2.18}
	\begin{cases}\Delta f = \lambda e^f (e^{af} - 1) + 4\pi \sum\limits_{j=1}^M n_j \delta_{p_j}, & \text{in } \Omega, \\
		f(x) = 0, & \text{on } \partial \Omega.\end{cases}
	\end{equation}
	There exists a maximal solution \( f_{\Omega} \) such that, for any other solution \( \tilde{f}_{\Omega} \) of Eq.~\eqref{2.18}, one has \( \tilde{f}_{\Omega} \leq f_{\Omega} \) pointwise in \( \Omega \).
	
\end{lemma}

\begin{proof}
	
From Lemmas \ref{lemma 2.1} and \ref{le2.4}, and the fact that $L^2(\Omega)$ is finite-dimensional, it follows that

$$
f_k \rightarrow f_{\Omega} \quad \text {in  }  L^{2}(\Omega),
$$
and
\begin{equation}\label{2.19}
\left\|f_{\Omega}\right\|_{L^{2}(\Omega)} \leq C.
\end{equation}
Since \( f_k \) converges pointwise to \( f_{\Omega} \), it follows that \( f_{\Omega} \) satisfies Eq.~\eqref{2.18}.

Suppose now that \( \tilde{f}_{\Omega} \) is another solution to Eq.~\eqref{2.18}. Assume that there exists a point \( x_0 \in \Omega \) such that
\[
\tilde{f}_{\Omega}(x_0) = \sup_{x \in \Omega} \tilde{f}_{\Omega}(x) > 0.
\]
Then we obtain
\[
0 \geqslant \Delta \tilde{f}_{\Omega}(x_0) \geqslant \lambda e^{\tilde{f}_{\Omega}} \left( e^{a\tilde{f}_{\Omega}} - 1 \right) > 0,
\]
which leads to a contradiction. Therefore, \( \tilde{f}_{\Omega}(x) \leq 0 = f_0 \) holds for all \( x \in \Omega \).

Next, assume that \( \tilde{f}_{\Omega} \leq f_k \). Then we compute
\[
\begin{aligned}
	(\Delta - K)(\tilde{f}_{\Omega} - f_{k+1}) &= \lambda e^{\tilde{f}_{\Omega}}(e^{a \tilde{f}_{\Omega}} - 1) - \lambda e^{f_k}(e^{a f_k} - 1) - K(\tilde{f}_{\Omega} - f_k) \\
	&\geqslant \lambda e^{\xi}[(a+1)e^{a\xi} - 1](\tilde{f}_{\Omega} - f_k) - K(\tilde{f}_{\Omega} - f_k) \\
	&\geqslant (\lambda a - K)(\tilde{f}_{\Omega} - f_k) \\
	&\geqslant 0,
\end{aligned}
\]
where \( \xi \) lies between \( \tilde{f}_{\Omega} \) and \( f_k \).

By the maximum principle (Lemma 2.2 in~\cite{hua2023existence}), we conclude that
\[
\tilde{f}_{\Omega} \leq f_{k+1}.
\]
Using mathematical induction,  we deduce that
\[
\tilde{f}_{\Omega} \leq f_{k}\quad \text{for all }  k,
\]
which implies
\[
\tilde{f}_{\Omega} \leq f_{\Omega}.
\]
Therefore, \( f_{\Omega} \) is the maximal solution.

\end{proof}

Based on the preceding results, we are in a position to prove Theorem~\ref{th1}. 
To this end, consider an increasing family of bounded and connected domains \( \{\Omega_n\} \) satisfying

\[
\Omega_0 \subset \Omega_n \subset \Omega_{n+1} \quad \text{for all } n \in \mathbb{N},
\]
and
\[
\bigcup_{n=1}^{\infty} \Omega_n = V.
\]
This exhaustion allows us to construct solutions on expanding domains and then pass to the limit as \( n \to \infty \). Let \( f^{(n)} \) denote the maximal solution of the following equation:
\begin{equation*}
	\begin{cases}\Delta f = \lambda e^f (e^{af} - 1) + 4\pi \sum\limits _{j=1}^M n_j \delta_{p_j}, & \text{in } \Omega_n, \\
		f(x) = 0, & \text{on } \partial \Omega_n.\end{cases}
\end{equation*}
Note that \( f^{(n+1)} \leq 0 \) on \( \bar{\Omega}_n \). Following the proof of Lemma~\ref{le2.5}, we obtain that
\( f^{(n+1)} \leq f^{(n)} \) on \( \bar{\Omega}_n \). Define the extension of \( f^{(n)} \) to the entire domain \( V \) as
\begin{equation*}
	\tilde{f}^{(n)}(x) = 
	\begin{cases}
		f^{(n)}(x), & \text{if } x \in \Omega_n, \\
		0, & \text{if } x \in V \setminus \Omega_n.
	\end{cases}
\end{equation*}
Then we conclude that
\[
0 \geq \tilde{f}^{(1)} \geq \tilde{f}^{(2)} \geq \cdots \geq \tilde{f}^{(n)} \geq \cdots.
\]
It follows from~(\ref{2.19}) that
\[
\| \tilde{f}^{(n)}\|_{L^{2}(V)} \leq C.
\]

Hence, we conclude that, up to a subsequence, \( \tilde{f}^{(n)} \) converges to a limit \( \tilde{f} \),
 which satisfies
\[
\Delta f = \lambda e^f (e^{a f} - 1) + 4\pi \sum_{j=1}^M n_j \delta_{p_j} \quad \text{in } V,
\]
with the condition
\[
\lim_{d(x) \to +\infty} \tilde{f}(x) = 0.
\]

We proceed to establish that \( \tilde{f} \) is indeed the maximal solution.
 Suppose \( \hat{f} \) is another topological solution to Eq.~\eqref{1.2}. Suppose there is a point \( x \in V \) with \( \hat{f}(x) > 0 \).
  Then there exists a domain \( \Omega_{n_0} \) and a point \( \bar{x} \in \Omega_{n_0} \) such that
\[
\hat{f}(\bar{x}) = \sup_{x \in \Omega_{n_0}} \hat{f}(x) > 0.
\]
At the point \( \bar{x} \) where \( \hat{f} \) attains its maximum, one obtains
\[
0 \geq \Delta \hat{f}(\bar{x}) = \lambda e^{\hat{f}(\bar{x})}\!\left(e^{a\hat{f}(\bar{x})} - 1\right) + 4\pi \sum_{j=1}^M n_j \delta_{p_j} > 0,
\]
a contradiction. Consequently, \( \hat{f}(x) \le 0 \) for every \( x \in V \).

Proceeding as in the proof of Lemma~\ref{le2.5}, we arrive at 
\[
\hat{f}(x) \leq f^{(n)}(x) \quad \text{in } \Omega_n, \quad \text{for all } n \in \mathbb{N}.
\]
Taking the limit, we deduce that
\[
\hat{f}(x) \leq \tilde{f}(x) \quad \text{for all } x \in V.
\]
Therefore, \( \tilde{f} \) is the maximal topological solution.

We now turn to the analysis of the decay behavior of $\tilde{f}$.  Note that
\[
\lim_{d(x) \to +\infty} \tilde{f}(x) = 0.
\]
Let \( \varepsilon\in(0,1) \) be fixed. Then for sufficiently large \( R>1 \),
\[
\tilde{\Omega}:=\{x\in V: d(x)\ge R\}\subset \Omega_0^{c},
\]
and, throughout \( \tilde{\Omega} \),
\[
\lambda a e^{(a+1)\tilde{f}(x)} \ge 2n\left[\!\left(1+\frac{\lambda a}{2n}\right)^{1-\varepsilon}-1\!\right].
\]

Hence, on $\tilde{\Omega}$, there holds 
$$
\Delta \tilde{f}=\lambda e^{\tilde{f}}\left(e^{a\tilde{f}}-1\right)=\lambda ae^{\tilde{f}+a\xi} \tilde{f} \leqslant \lambda ae^{(a+1) \tilde{f}} \tilde{f}\leqslant c_1 \tilde{f},
$$
where $\xi$ lies between $\tilde{f}$ and $0$, $c_1= 2 n\left[\left(1+\frac{\lambda a}{2 n}\right)^{1-\epsilon}-1\right]$.

For any \( \varepsilon \in (0,1) \), let
\[
v(x) := -e^{-\alpha(1 - \varepsilon)d(x)}
\]
be the auxiliary function,
where \( \alpha = \ln\left(1 + \frac{\lambda a}{2n} \right) \) and \( 0 < \varepsilon < 1 \). Let \( e_i \in \mathbb{R}^n \) denote the \( i \)-th canonical basis vector, that is,
\[
e_i = (0,\, 0,\, \dots,\, 1,\, \dots,\, 0)^\top.
\]

On the domain \( \tilde{\Omega} \), letting \( s = d(x) \), the discrete Laplacian of \( v \) is given by
\[
\Delta v(x) = \sum_{y \sim x} \big( v(y) - v(x) \big) = \sum_{i=1}^n \left( v(x + e_i) + v(x - e_i) - 2v(x) \right).
\]

If \( x_i \neq 0 \), then
\[
v(x + e_i) + v(x - e_i) - 2v(x) = -e^{-\alpha(1 - \varepsilon)(s - 1)} - e^{-\alpha(1 - \varepsilon)(s + 1)} + 2e^{-\alpha(1 - \varepsilon)s}.
\]

If \( x_i = 0 \), we have
\[
\begin{aligned}
	v(x + e_i) + v(x - e_i) - 2v(x) &= -2e^{-\alpha(1 - \varepsilon)(s + 1)} + 2e^{-\alpha(1 - \varepsilon)s} \\
	&\geq -e^{-\alpha(1 - \varepsilon)(s - 1)} - e^{-\alpha(1 - \varepsilon)(s + 1)} + 2e^{-\alpha(1 - \varepsilon)s}.
\end{aligned}
\]

To estimate $\Delta v(x)$, we observe that:
$$
\begin{aligned}
	\Delta v(x) & \geqslant n\left[-e^{-\alpha(1-\varepsilon)(s-1)}-e^{-\alpha(1-\varepsilon)(s+1)}+2 e^{-\alpha(1-\varepsilon) s}\right] \\
	& =n\left[e^{-\alpha(1-\varepsilon)}+e^{\alpha(1-\varepsilon)}-2\right] v(x) \\
	& =n\left[\left(1+\frac{c_1}{2 n}\right)+\frac{1}{1+\frac{c_1}{2 n}}-2\right] v(x) \\
	& \geqslant n\left[2\left(1+\frac{c_1}{2 n}\right)-2\right] v(x)\\
	&=c_1 v(x).
\end{aligned}
$$
Denote by \( C(\varepsilon) \) a constant determined only by \( \varepsilon \).  Then, on \( \tilde{\Omega} \), we have
\begin{equation}\label{eq:decay-est}
	\left( \Delta - c_1 \right) \left( C(\varepsilon) v - \tilde{f} \right)
	\geq c_1 C(\varepsilon) v - c_1 C(\varepsilon) - c_1 \tilde{f} + c_1 \tilde{f} 
	= 0.
\end{equation}

Choose \( C(\varepsilon) \) sufficiently large such that
\[
C(\varepsilon) v(x) - \tilde{f}(x) \leq 0 \quad \text{for all } x \text{ with } d(x) = R.
\]
Note that
\[
\lim_{d(x) \to \infty} \left( C(\varepsilon) v(x) - \tilde{f}(x) \right) = 0.
\]

Assume, for contradiction, that there exists a point \( x \in \tilde{\Omega} \) with \( (C(\varepsilon)v - \tilde{f})(x) > 0 \).  
Then one can find a bounded domain
\[
\hat{\Omega} = \{ x \in V \mid R \le d(x) \le R_0 \}
\]
for some constant \( R_0 > R \), and a point \( x_0 \in \hat{\Omega} \) where
\[
(C(\varepsilon)v - \tilde{f})(x_0) = \sup_{x \in \hat{\Omega}} \big(C(\varepsilon)v(x) - \tilde{f}(x)\big) > 0.
\]

However, evaluating both sides of \eqref{eq:decay-est} at \( x_0 \), we obtain
\[
\Delta \left( C(\varepsilon) v - \tilde{f} \right)(x_0) \geq c_1 \left( C(\varepsilon) v - \tilde{f} \right)(x_0) > 0,
\]
which contradicts the maximality of \( x_0 \), since the discrete Laplacian at a maximum point should satisfy
\[
\Delta \left(C(\varepsilon)v - \tilde{f}\right)(x_0) \le 0.
\]
This contradiction implies that the function \( C(\varepsilon)v - \tilde{f} \) cannot attain a positive maximum in \( \hat{\Omega} \), and hence must be non-positive throughout $\tilde{\Omega}$.

Therefore, we conclude that
\[
0 \geq \tilde{f}(x) \geq - C(\varepsilon) e^{ -\alpha(1 - \varepsilon) d(x) },
\]
and hence \( \tilde{f} \) belongs to \( L^p(V) \) for all \( p \ge 1 \).
\section{Proof of Theorem 1.2}

First, we denote by $f_{\lambda}$ the topological solution obtained in Theorem~\ref{th1}. Note that \[
f_{\lambda}(x) = O\left(e^{-\alpha(1 - \varepsilon) d(x)}\right).
\]  By applying Green's formula, we obtain
\[\lambda \sum_{x \in \mathbb{Z}^n} e^{f_\lambda}\left(1-e^{af_\lambda}\right)=\sum_{x \in \mathbb{Z}^n} g(x)=N.
\]

Thus we get
\begin{equation}\label{3.1}
e^{f_\lambda}\bigl(1-e^{a f_\lambda}\bigr)\le \frac{N}{\lambda}.
\end{equation}

Let \( t = e^{f_\lambda} \). Since \( f_\lambda \le 0 \), we have \( 0 < t \le 1 \).  
Then (\ref{3.1}) then becomes
\[
t(1-t^a)\le \frac{N}{\lambda}.
\]

Define
\[
\eta(t)=t(1-t^a)=t-t^{a+1}, \qquad t\in(0,1].
\]
Then
\[
\eta'(t)=1-(a+1)t^{a}.
\]
Setting \( \eta '(t)=0 \) yields the critical point
\[
t_0=(a+1)^{-1/a}.
\]

Since \( \eta'(t) > 0 \) for \( 0<t<t_0 \) and \( \eta'(t)<0 \) for \( t_0<t<1 \),  
the function \( \eta(t) \) is increasing on \( (0,t_0] \) and decreasing on \( [t_0,1) \).  
Hence \( \eta(t) \) attains its maximum at \( t=t_0 \), and
\[
\eta(t_0)
= t_0 - t_0^{a+1}
= (a+1)^{-1/a} - (a+1)^{-(a+1)/a}.
\]
Denote this maximum value by \( \eta_0 \).

If
\[
\frac{N}{\lambda} < \eta_0,
\]
then the equation
\[
t(1-t^{a})=\frac{N}{\lambda}
\]
has exactly two solutions \( t_1, t_2 \) satisfying
\[
0<t_1<t_0<t_2<1.
\]
From the inequality
\[
t(1-t^{a}) \le \frac{N}{\lambda},
\]
we know that, when \(\lambda\) is sufficiently large, the values of \(t=e^{f_\lambda(x)}\)
must satisfy either
\[
t \le t_1 \qquad\text{or}\qquad t \ge t_2.
\]

We claim that, for all sufficiently large $\lambda$, the only admissible branch is
\[
t(x)\ge t_2 .
\]
Equivalently, $f_\lambda(x)\ge \ln t_2$. 
Suppose to the contrary that there exists some $x\in\mathbb{Z}^n$ such that
\[
f_\lambda(x)\le \ln t_1 .
\]
Since \(\lim_{d(x)\to\infty} f_\lambda(x)=0\) and the lattice \(\mathbb{Z}^n\) is connected, we may find two neighboring points \(x_1\sim x_2\) such that
\[
f_\lambda(x_1) \ge \ln t_2, \qquad 
f_\lambda(x_2) \le \ln t_1 .
\]

Using the definition of the  discrete Laplacian, we have 
\[
\Delta f_\lambda(x_1)
\le f_\lambda(x_2) - 2n\, f_\lambda(x_1).
\]
Note that 
\[
\Delta f_\lambda(x)
= g(x) + \lambda e^{f_\lambda(x)}\bigl(e^{a f_\lambda(x)}-1\bigr),
\]
and 
\[
\bigl|\Delta f_\lambda(x_1)\bigr|
\le |g(x_1)| + \lambda e^{f_\lambda(x_1)}\bigl|e^{a f_\lambda(x_1)}-1\bigr|
\le N + N = 2N.
\]

We conclude that 
\[
2n\, f_\lambda(x_1)
\le f_\lambda(x_2) + 2N.
\]
Substituting \(f_\lambda(x_1)\ge \ln t_2\) and \(f_\lambda(x_2)\le \ln t_1\), we obtain
\begin{equation}\label{lni}
2n \ln t_2 \le \ln t_1 + 2N.
\end{equation}
Noting that 
\[
\lim_{\lambda\to\infty}\ln t_2 = 0
\quad\text{and}\quad
\lim_{\lambda\to\infty}\ln t_1 = -\infty,
\]
we see that the left-hand side \(2n\ln t_2\) tends to \(0\), while the right-hand side 
\(\ln t_1 + 2N\) tends to \(-\infty\) as $\lambda\rightarrow \infty$.
Therefore, there exists a constant \(\lambda_0>0\) such that, for all 
\(\lambda>\lambda_0\),
\[
2n \ln t_2 > \ln t_1 + 2N.
\]
This contradicts the inequality \eqref{lni}, so the branch \(t\le t_1\) is impossible when \(\lambda>\lambda_0\). Hence we conclude that
\[
f_\lambda(x) \ge \ln t_2, \qquad \forall x\in\mathbb{Z}^n, \quad \text{if } \lambda>\lambda_0.
\]

For $0 < \lambda_1 < \lambda_2$, we have
\[
\Delta f_{\lambda_1} 
= \lambda_1 e^{f_{\lambda_1}}\left(e^{a f_{\lambda_1}} - 1\right)+g
\geq \lambda_2 e^{f_{\lambda_1}}\left(e^{a f_{\lambda_1}} - 1\right)+g.
\]
Let $\{f_{\lambda_2}^k\}$ be the sequence generated by the iteration scheme \eqref{2.1} with $\lambda=\lambda_2$ on $\Omega$. 
Arguing as in the proof of Lemma~\ref{le2.5}, we obtain
\[
f_{\lambda_1}\le f_{\lambda_2}^k \qquad \text{for all } k\ge 0.
\]
It follows that 
\[
f_{\lambda_1}\le f_{\lambda_2}.
\]

In the following,  we discuss two separate cases: $n = 2$ and $n \geq 3$.

In the case $n = 2$, we assume that there exists a function $f$ such that 
$$f_{\lambda}(x) \rightarrow f(x)\ \text{as}\ \lambda \rightarrow 0_{+}.$$
Then, we get 
$$
\begin{cases}\Delta f(x)=g(x) & \forall x \in \mathbb{Z}^2, \\ f(x)<0 & \forall x \in \mathbb{Z}^2 .\end{cases}
$$

Let us introduce, for each $d \in \mathbb{N}_{+}$,
\[
Q_d=\left\{(x_1,x_2)\in\mathbb{Z}^2:\ |x_1|\le d,\ |x_2|\le d \right\}.
\]

Choose $d_0 \in \mathbb{N}_{+}$ such that $\Omega_0 \subset Q_{d_0}$. 
For any $d > d_0$, denote by $u_d$ the unique solution of
\[
\begin{cases}
	\Delta u_d = g & \text{in } Q_d,\\[2mm]
	u_d = 0        & \text{on } \partial Q_d .
\end{cases}
\]
By construction, $u_d$ satisfies
\[
f\le u_d \le 0 \qquad \text{in } Q_d .
\]
It follows that
\begin{equation}\label{ti1}
	\begin{aligned}
		-4\pi\sum_{j=1}^{M} n_j\, u_d(p_j)
		&= -\int_{Q_d} (\Delta u_d)\,u_d
		= \int_{\bar{Q}_d} |\nabla u_d|^2 \, d\mu \\
		&= \frac{1}{2}\sum_{\substack{x,y\in\bar{Q}_d\\ x\sim y}}
		\bigl(u_d(y)-u_d(x)\bigr)^2 \\
		&= \frac{1}{2}\sum_{\substack{x,y\in Q_d\\ x\sim y}}
		\bigl(u_d(y)-u_d(x)\bigr)^2
		+ \sum_{\substack{x\in Q_d,\, y\in\partial Q_d\\ x\sim y}}
		\bigl(u_d(y)-u_d(x)\bigr)^2 \\
		&\ge \sum_{i = d_0}^{d}
		\sum_{\substack{x\in Q_i,\, y\in\partial Q_i\\ x\sim y}}
		\bigl(u_d(y)-u_d(x)\bigr)^2 .
	\end{aligned}
\end{equation}

Note that
\[
\sum_{x \in Q_i} \Delta u_d(x)
= \sum_{x \in Q_i} \sum_{y \sim x} \bigl(u_d(y)-u_d(x)\bigr)
= \sum_{\substack{x\in Q_i,\, y\in\partial Q_i\\ x \sim y}}
\bigl(u_d(y)-u_d(x)\bigr),
\]
since all contributions from interior edges cancel pairwise. 
By the Cauchy inequality, we obtain
\begin{equation}\label{ti2}
	\left(\sum_{x \in Q_i} \Delta u_d(x)\right)^{2}
	\leq  |\partial Q_i|\!
	\sum_{\substack{x\in Q_i,\, y\in\partial Q_i\\ x \sim y}}
	\bigl(u_d(y)-u_d(x)\bigr)^2
	= (8i+4)
	\sum_{\substack{x\in Q_i,\, y\in\partial Q_i\\ x \sim y}}
	\bigl(u_d(y)-u_d(x)\bigr)^2 .
\end{equation}
Combining \eqref{ti1} and \eqref{ti2}, we conclude that
\[
-4\pi\sum_{j=1}^{M} n_j\, u_d(p_j)
\geqslant \sum_{i = d_0}^{d}
\frac{1}{8 i+4}\left(\int_{Q_i} \Delta u_d\right)^{2}
= N^{2}
\sum_{i = d_0}^{d} \frac{1}{8 i+4}.
\]

Hence,
\[
4\pi\sum_{j=1}^{M} n_j\, u_d(p_j)
\longrightarrow -\infty
\qquad \text{as } d\to+\infty,
\]
which contradicts  \(f \le u_d \le 0\).

From the previous estimates, we can  select a point $x_0\in\mathbb{Z}^2$ satisfying 
\[
\lim_{\lambda\to 0^{+}} f_\lambda(x_0) = -\infty .
\]

For any neighbour $y$ of $x_0$ (that is, $y\sim x_0$), observe that the discrete Laplacian at $y$ satisfies
\begin{equation}\label{3.4}
\Delta f_\lambda(y)
\le f_\lambda(x_0)- 2n\, f_\lambda(y).
\end{equation}

Moreover, using Eq.~(\ref{1.2}), we have the uniform bound
\begin{equation}\label{3.5}
|\Delta f_\lambda(x)|
\le |g(x)| + \lambda e^{f_\lambda(x)}\bigl(1-e^{a f_\lambda(x)}\bigr)
\le N+\lambda \sum_{x\in\mathbb{Z}^n} e^{f_\lambda(x)}\bigl(1-e^{a f_\lambda(x)}\bigr)
\le 2N ,
\end{equation}
 for all $x\in\mathbb{Z}^n$.

Combining (\ref{3.4}) and (\ref{3.5}) gives
\[
f_\lambda(y)
\le \frac{1}{2n}\Bigl( f_\lambda(x_0) + 2N \Bigr),
\]
and therefore
\[
\lim_{\lambda\to 0^{+}} f_\lambda(y)
= -\infty .
\]

Since the lattice $\mathbb{Z}^2$ is connected, iterating this argument along any finite path implies
\[
f_\lambda(x)\longrightarrow -\infty
\qquad\text{for every } x\in\mathbb{Z}^2,
\quad \text{as } \lambda\to 0^{+}.
\]

We now turn to the case $n=3$. Set
\[
\phi_n(x) := 4\pi \sum_{j=1}^{M} n_j G_n(x-p_j).
\]
The function $\phi_n$ is characterized by
\[
\begin{cases}
	\Delta \phi_n = 4\pi \displaystyle\sum_{j=1}^{M} n_j\, \delta_{p_j} = g 
	& \text{on } \mathbb{Z}^n,\\[2mm]
	\phi_n(x) \to 0 
	& \text{as } d(x)\to+\infty .
\end{cases}
\]
Observe that 
\begin{equation}\label{3.6}
\Delta(\phi_n - f_\lambda) \ge 0
\end{equation} and \[
\lim_{d(x)\to+\infty} \bigl(\phi_n - f_\lambda\bigr)(x) = 0.
\]
We claim that
\begin{equation}\label{3.7}
\phi_n(x) \le f_\lambda(x) \le 0, 
\qquad \forall\, x \in \mathbb{Z}^n .
\end{equation}
In fact, if this is false, then there must exist some point $x_0$ such that $\phi_n - f_\lambda$ attain a positive maximum at $x_0$. Hence, 
\[\Delta(\phi_n - f_\lambda)(x_0)<0,\] which contradicts (\ref{3.6}). 
Consequently,
\[
\lim_{\lambda\rightarrow 0^+}f_\lambda(x) =f(x) ,\; \forall\, x\in\mathbb{Z}^n,
\]
and $f$ satisfies 
\[
\phi_n(x) \le f(x) \le 0, \qquad \forall\, x\in\mathbb{Z}^n.
\]
Observe that 
\[
	\Delta(\phi_n - f) = 0
\] and \[
\lim_{d(x)\to+\infty} (\phi_n - f )= 0.
\]
Proceeding as in the derivation of \eqref{3.7}, we arrive at $f(x)=\phi_n(x)$ for all $x\in \mathbb{Z}^n$.

\vskip 30 pt
\noindent{\bf ACKNOWLEDGMENTS}

This work is   partially  supported by  the National Key Research and Development Program of China 2020YFA0713100.

\bibliographystyle{plain}
\bibliography{D:/Reference/CHS.bib}	
	
\end{document}